\documentclass{amsart}
\usepackage{ amsmath, amsthm, amsfonts, hyperref, graphicx, ifpdf}
\usepackage[dvipsnames,usenames]{color}
\hypersetup{
   bookmarks=true,         
   unicode=false,          
   pdftoolbar=true,        
   pdfmenubar=true,        
   pdffitwindow=false,     
   pdfstartview={FitH},    
   pdftitle={My title},    
   pdfauthor={Author},     
   pdfsubject={Subject},   
   pdfcreator={Creator},   
   pdfproducer={Producer}, 
   pdfkeywords={keywords}, 
   pdfnewwindow=true,      
   colorlinks=true,       
   linkcolor=blue,          
   citecolor=blue,        
   filecolor=magenta,      
   urlcolor=MidnightBlue          
}

\newcommand{\ve}{{\bf e}}

\newcommand{\vv}{{\bf v}}

\newcommand{\vz}{{\bf z}}

\newcommand{\VP}{{\varphi}}
\newcommand{\VE}{{\varepsilon}}
\newcommand{\vnu}{\boldsymbol{\nu}}

\newcommand{\Ric}{\operatorname{Ric}}

\newcommand{\Div}{\operatorname{div}}

\newtheorem{theorem}{Theorem}[section]
\newtheorem{lemma}[theorem]{Lemma}
\newtheorem{proposition}[theorem]{Proposition}

\theoremstyle{definition}

\theoremstyle{remark}
\newtheorem{remark}[theorem]{Remark}

\numberwithin{equation}{section}

\begin{document}
\setlength{\baselineskip}{1.2\baselineskip}

\title[MMCF of entire locally Lipschitz radial graphs in hyperbolic space]
{Modified mean curvature flow of entire locally Lipschitz radial graphs in hyperbolic space}

\author{Patrick Allmann}
\address{Mathematics Department\\University of California, Santa Cruz, 1156 High Street\\ Santa Cruz, CA 95064\\USA}
\email{pallmann@ucsc.edu}

\author{Longzhi Lin}
\address{Mathematics Department\\University of California, Santa Cruz, 1156 High Street\\ Santa Cruz, CA 95064\\USA}
\email{lzlin@ucsc.edu}

\author{Jingyong Zhu}
\address{Max Planck Institute for Mathematics in the Sciences, \\ Inselstrasse 22, 04103 Leipzig, Germany}
\email{jizhu@mis.mpg.de}



\begin{abstract}
The \textit{Asymptotic Plateau Problem} asks for the existence of smooth
complete hypersurfaces of constant mean curvature with prescribed asymptotic boundary at infinity in the hyperbolic space $\mathbb{H}^{n+1}$. The modified
mean curvature flow (MMCF)
$$\frac{\partial \mathbf{F}}{\partial t} = (H-\sigma)\vnu, \quad \quad \sigma\in (-n,n)$$ 
was firstly introduced by Xiao and the second author a few years back in \cite{LX12}, and it provides a tool using geometric flow to find such hypersurfaces with constant mean curvature in $\mathbb{H}^{n+1}$. Similar to the usual mean curvature flow, the MMCF is the natural negative $L^2$-gradient flow of the area-volume functional $\mathcal{I}(\Sigma)=A(\Sigma)+\sigma V(\Sigma)$ associated
to a hypersurface $\Sigma$. In this paper, we prove that the MMCF starting from
an entire $locally$ $Lipschitz$ continuous radial graph exists and stays radially
graphic for all time. In general one cannot expect the convergence of the flow as it can be seen from the flow starting from a horosphere (whose asymptotic boundary is degenerate to a point). 
\end{abstract}

\maketitle

\section{Introduction} 

Mean curvature flow (MCF) was first studied by Brakke \cite{B78} in the context
of geometric measure theory. Later, smooth compact surfaces evolved by MCF
in Euclidean space were investigated by Huisken in \cite{H84} and \cite{H90}, and in
arbitrary ambient manifolds in \cite{H86}. The evolution of entire graphs by MCF in $\mathbb{R}^{n+1}$ was also studied in \cite{EH89}, the result being improved in \cite{EH91}. Lately, the MCF in Euclidean space has attracted much attention. See, e.g., the survey of various aspects of the MCF of hypersurfaces by Colding, Minicozzi and Pedersen \cite{CMP} and the references therein. In \cite{U03}, Unterberger considered the MCF in hyperbolic space $\mathbb{H}^{n+1}$ and proved that if the initial surface $\Sigma_0$ has bounded hyperbolic height over ${\mathbb{S}}^n_+$,  (i.e., $\partial \Sigma_0 = \partial \mathbb{S}^n_+$), then under the MCF, $\Sigma_t$ converges in $C^{\infty}$ to ${\mathbb{S}}^n_+$, which is minimal.

The Asymptotic Plateau Problem of finding smooth complete hypersurfaces of constant mean curvature in hyperbolic space $\mathbb{H}^{n+1}$ with prescribed asymptotic boundary at infinity has also been studied over the years, see \cite{A82}, \cite{HL87}, \cite{Lin89}, \cite{T96} and \cite{NS96}. In \cite{GS00} Guan and Spruck proved the existence and uniqueness of smooth complete hypersurfaces of constant mean curvature $\sigma\in (-n,n)$ in hyperbolic space with prescribed $C^{1,1}$ star-shaped asymptotic boundary at infinity. In \cite{DS09}, among others, De Silva and Spruck recovered this result using the method of calculus of variations. In the previous joint work \cite{LX12} of Xiao and the second author, the following modified mean curvature flow (MMCF) was first introduced, which is the natural negative $L^2$-gradient flow of the area-volume functional $\mathcal{I}(\Sigma)\,=\, \mathcal{I}_{\Omega} (v)\, =\,A_{\Omega}(v) + \sigma V_{\Omega}(v)$ associated to $\Sigma$ as in \cite{DS09}. It can be used to continuously deform hypersurfaces in $\mathbb{H}^{n+1}$ into constant mean curvature hypersurfaces with prescribed asymptotic boundary at infinity.

Let $\mathbf{F}(\mathbf{z},t): {\mathbb{S}}^n_{+} \times [0,\infty) \rightarrow \mathbb{H}^{n+1}$ be the complete embedded star-shaped hypersurfaces (as complete radial graphs over ${\mathbb{S}}^n_{+}$) moving by the MMCF in hyperbolic space $\mathbb{H}^{n+1}$, where ${\mathbb{S}}^n_+$ is the upper hemisphere of the unit sphere ${\mathbb{S}}^n$ in $\mathbb{R}^{n+1}$ and the half-space model of $\mathbb{H}^{n+1}$ is used. That is, $\mathbf{F}(\cdot,t)$ is a one-parameter family of smooth immersions with images $\Sigma_{t}=\mathbf{F}({\mathbb{S}}^n_+, t),$ satisfying the evolution equation
\begin{equation}\label{MMCF0}
\left\{
\begin{aligned}
\frac{\partial}{\partial t} \mathbf{F}(\mathbf{z},t) &= (H-\sigma)\vnu_H\,,\quad (\mathbf{z},t) \in {\mathbb{S}}^n_+\times [0,\infty)\,,\\
\mathbf{F}(\mathbf{z},0) &= \Sigma_0\,,\quad \mathbf{z} \in {\mathbb{S}}^n_+ \,,\\
\mathbf{F}(\mathbf{z},t) &= \Gamma \,,\quad \mathbf{z} \in \partial{\mathbb{S}}^n_+ \,,
\end{aligned}
\right.
\end{equation}
where $H = \sum_{i=1}^n \kappa^H_i$ denotes the hyperbolic mean curvature of $\Sigma_t$, $\sigma \in (-n,n)$ is a constant, and
$\vnu_H$ denotes the outward unit normal of $\Sigma_t$ with respect to the hyperbolic metric. More precisely, suppose the solution $\mathbf{F}(\mathbf{z},t)$ to the MMCF \eqref{MMCF0} can be represented as a complete radial graph over ${\mathbb{S}}^n_+$. That is,
\begin{equation}\label{radialheight}
\mathbf{F}(\mathbf{z},t)\,=\, x(\mathbf{z},t)\,=\,e^{v(\mathbf{z},t)}\mathbf{z}\,, \quad (\mathbf{z},t) \in  {\mathbb{S}}^n_+ \times (0,\infty)\,,
\end{equation}
and $\Gamma \subset \partial_\infty \mathbb{H}^{n+1} = \{x_{n+1} = 0\}$ is the radial graph of a function $e^{\phi}$ over $\partial {\mathbb{S}}^n_+$, i.e., $\Gamma$ can be represented by
$$ \Gamma(\mathbf{z})\,=\, e^{\phi(\mathbf{z})}\mathbf{z}\,,\quad \mathbf{z}\in \partial {\mathbb{S}}^n_+\,.$$
We call such a function $v(\mathbf{z},t)$ the radial height of $\Sigma_t = \mathbf{F}(\cdot, t)$. Note that $\Sigma_t$ remains a radial graph as long as the support function $\langle\vnu_E,x\rangle_E$ satisfies
\begin{equation}\label{RadialG}
\langle\vnu_E,x\rangle_E \,>\,0\,,
\end{equation}
where $\vnu_E$ is the Euclidean outward unit normal vector of $\Sigma_t$\,.
Then one observes that the Cauchy initial-boundary value problem for the MMCF \eqref{MMCF0} is equivalent to the following degenerate parabolic PDE with initial and boundary conditions:
\begin{equation}\label{MMCF1}
\left\{
\begin{aligned}\frac{\partial v(\mathbf{z},t)}{\partial t} &=\, y^2\frac{\alpha^{ij}v_{ij}}{n}-y\mathbf{e}\cdot\nabla
v-\sigma yw \,, \quad (\mathbf{z},t)\in {\mathbb{S}}^n_+ \times (0,\infty)\,,\\
v(\mathbf{z},0) &= v_0(\mathbf{z})\,,\quad \mathbf{z} \in {\mathbb{S}}^n_+\,,\\
v(\mathbf{z},t) &= \phi(\mathbf{z})\,,\quad (\mathbf{z},t)\in \partial {\mathbb{S}}^n_+ \times [0,\infty)\,,
\end{aligned}
\right.
\end{equation}
where we represent $\Sigma_0$ as the radial graph of the function $e^{v_0}$ over $\mathbb{S}^n_+$ and $v_0\big|_{\partial {\mathbb{S}}^n_+} = \phi$\,. Here $y= \langle{\mathbf{e}}, {\bf{z}}\rangle_E$, and $\mathbf{e}$ is the unit vector in the positive $x_{n+1}$ direction in $\mathbb{R}^{n+1}$. Also, $\alpha^{ij}=\gamma^{ij}-\frac{\gamma^{ik}v_kv_j}{w^2} \,,1\leq i,j\leq n$, $w=(1+|\nabla v|^2)^{1/2}$ and we denote by $\gamma_{ij}$ the standard metric of ${\mathbb{S}^n_+}$ and $\gamma^{ij}$ its inverse. Note that the MCF, i.e., the case of $\sigma =0$ for \eqref{MMCF0} was considered in \cite{U03}, but the case of $\sigma \neq 0$ is substantially different, see Remark \ref{obstruction}.
\vskip 2mm

In \cite{LX12}, the Cauchy initial-boundary value problem \eqref{MMCF1} for the MMCF of complete radial graphs was studied. The flow starting from an entire star-shaped Lipschitz continuous radial graph with the \textit{uniform local ball condition} on the asymptotic boundary was shown to exist for all time and converge to a complete hypersurface of constant mean curvature with prescribed asymptotic boundary at infinity. Let us elaborate a bit on the \textit{uniform local ball condition}. Due to the degeneracy at infinity of the MMCF \eqref{MMCF1} for radial graphs, we will use the method of continuity and consider the approximate problem.  For fixed $\epsilon>0$ sufficiently small, let $\Gamma_{\epsilon}$ be the vertical translation of $\Gamma \subset \{x_{n+1} = 0\}$ to the plane $\{x_{n+1} = \epsilon\}$ and let $\Omega_{\epsilon}$ be the subdomain of ${\mathbb{S}}^n_+$ such that $\Gamma_{\epsilon}$ is the radial graph over $\partial \Omega_{\epsilon}$ (see Figure \ref{PIC1}). \input{P1.TpX} For any $\epsilon\geq 0$ sufficiently small and any point $P \in \partial \Sigma_0^{\epsilon} = \Gamma_{\epsilon}$ (denoting $\Sigma_0^0 = \Sigma_0$ and $\Gamma_0 = \Gamma$), the uniform star-shapedness of $\Gamma_{\epsilon}$ implies that there exist balls $B_{R_1}(a, P)$ and $B_{R_2}(b,P)$ with radii $R_1>0$ and $R_2>0$ and centered at $a = (a', -\sigma R_1)$ and $b = (b', \sigma R_2)$, respectively, such that $\{x_{n+1} = \epsilon\} \cap B_{R_1}(a, P)$ is internally tangent to $\Gamma_{\epsilon}$ at $P$ and $\{x_{n+1} = \epsilon\} \cap B_{R_2}(b,P)$ is externally tangent to $\Gamma_{\epsilon}$ at $P$. $\partial B_{R_1}(a,P)$ and $\partial B_{R_2}(b,P)$ are the so-called \textit{equidistance spheres}. Note that in a small neighborhood $B_{\delta}(P)$ around $P$ for some $\delta>0$, both $\partial B_{R_1}(a,P)\cap B_{\delta}(P)$ and $\partial B_{R_2}(b, P)\cap B_{\delta}(P)$ can be locally represented as radial graphs. We say that the initial hypersurfaces $\Sigma_0^{\epsilon}$'s satisfy the uniform interior (resp. exterior) local ball condition whenever, for all $\epsilon \geq 0$ sufficiently small and all $P\in \Gamma_{\epsilon}$, we have $\Sigma_0^{\epsilon} \cap B_{\delta}(P)\cap B_{R_1}(a, P) = \{P\}$ (resp. $\Sigma_0^{\epsilon} \cap B_{\delta}(P)\cap B_{R_2}(b, P)= \{P\}$, see Figure \ref{PIC2}), \input{P2.TpX} and the local radial graph $\partial B_{R_1}(a, P)\cap B_{\delta}(P)$ (resp. $\partial B_{R_2}(b, P)\cap B_{\delta}(P)$) has a \textit{uniform} Lipschitz bound depending only on the star-shapedness of $\Gamma$. If the $\Sigma^{\epsilon}_0$'s satisfy both of the uniform interior and exterior local ball conditions, then we say $\Sigma_0$ satisfies the uniform local ball condition. Such a uniform gradient bound on the asymptotic boundary was necessary for a version of maximum principle to be applicable in order to obtain a global gradient bound, which ensures the long time existence and convergence of the flow.
\vskip 2mm

In this paper we would like to show the long time existence of the MMCF without the \textit{uniform local ball condition} at the infinity of the initial hypersurface. To this end, we consider the MMCF starting from an entire locally Lipschitz continuous radial graph $\Sigma_0 \subset \mathbb{H}^{n+1}$ and show the long time existence of the flow. More precisely, we prove
 \begin{theorem}\label{mainthem}
 Let $\mathbf{F}_0 : \mathbb{S}^n_+ \rightarrow \mathbb{H}^{n+1}$ be such that $\Sigma_0 = \mathbf{F}_0(\mathbb{S}^n_+)$ is an entire locally Lipschitz continuous radial graph over $\mathbb{S}^n_+$. Then the Cauchy initial-boundary value problem for the MMCF \eqref{MMCF0} has a solution $\mathbf{F}(\mathbf{z},t) \in C^\infty(\mathbb{S}_+^n\times (0,\infty))\cap C^{0+1,0+1/2}(\mathbb{S}^n_+\times [0,\infty))$ and $\mathbf{F}(\mathbf{\mathbb{S}^n_+},t)$ is a complete radial graph over $\mathbb{S}^n_+$ for any $t\geq 0$.
\end{theorem}

\begin{remark}
By the work of Guan-Spruck \cite{GS00}, Xiao and the second author \cite{LX12}, given a $C^{1,1}$ star-shaped  $n-1$ dimensional closed submanifold at the infinity $\partial_\infty \mathbb{H}^{n+1}$, we can find a suitable initial hypersurface such that the MMCF exists for all time and converges to a hypersurface of constant mean curvature which has the given submanifold as the asymptotic boundary. On the other hand, MMCF, starting from a horosphere $\{x_{n+1} = c\}$ (whose infinity is degenerate to a point in $\partial_\infty \mathbb{H}^{n+1}$), exists for all time but never converges. Given such an example, one cannot expect the full convergence of the flow, as it depends on the behavior of the initial asymptotic boundary. We expect that some intermediate geometric condition that is weaker (i.e., allows degeneracy of the initial asymptotic boundary to some extent) than the \textit{uniform local ball condition} in \cite{LX12} will guarantee the convergence of the flow. This will be investigated in our forthcoming paper.
\end{remark}

The paper is organized as follows. In Section \ref{prelim}, we fix some notation and review some necessary preliminary materials. In Section \ref{interior}, we use the evolution equation of the support function $\langle\vnu_E,x\rangle_E $ (see Proposition \ref{prop1}) and an appropriate space-time cut-off function together with a conventional maximum principle argument to show a uniform interior gradient estimate for the MMCF (see Theorem \ref{intgrad}). In Section \ref{higherderiv}, we show the interior estimates on all other higher order derivatives for the MMCF (see Theorem \ref{estonA2} and Theorem \ref{estonA3}). We prove the main Theorem \ref{mainthem} in Section \ref{proofofmain}.

\section{preliminary}\label{prelim}
Let's first fix some notation. Operators without subscripts or superscripts are operators on $\Sigma_t$. Corresponding operators  in hyperbolic space, Euclidean space, or on $\mathbb{S}^n_+$ will be denoted with either a subscript or a superscript $H, E, S$, respectively. Greek  indices will range from 1 to $n+1$, while Latin  indices will range from 1 to $n$.

Denote $ds^2_H$ by $\langle \cdot, \cdot \rangle_H$, and $\nabla^H$ the Levi-Civita connection on $\mathbb{H}^{n+1}$. The ambient Riemann curvature tensor with respect to the hyperbolic metric used in this paper is
$$(R^H)(X,Y)Z = \nabla^H_Y\nabla^H_XZ - \nabla^H_X\nabla^H_YZ + \nabla^H_{[X,Y]}Z.$$ 

Let $\{\ve_\alpha\}_{\alpha= 1}^{n+1}$ be the coordinate basis of $\mathbb{H}^{n+1}$ with respect to the standard coordinates $x^\alpha$ of $\mathbb{R}^{n+1}_+$. Define $(R^H)_{\alpha\beta\gamma\delta} = \langle(R^H)(\ve_\alpha,\ve_\beta)\ve_\gamma,\ve_\delta\rangle_H$, the components of the hyperbolic Riemann curvature tensor. And define the components of the hyperbolic Ricci tensor
\begin{equation}\label{Ric comp}
(\text{Ric}^H)_{\alpha\gamma} = (ds^2_H)^{\beta\delta}(R^H)_{\alpha\beta\gamma\delta},
\end{equation}
where $ (ds^2_H)^{\alpha\gamma}$ is the inverse of the metric $ds^2_H$.

Since the upper-half space model of hyperbolic space $\mathbb{H}^{n+1}$ and $\mathbb{R}^{n+1}_+$ are conformal, we have
\begin{proposition}\label{prop2.1}
For any two vector fields $X,Y$ on $\mathbb{H}^{n+1}$,
$$\nabla^H_X Y = \nabla^E_X Y + \frac{1}{x_{n+1}}(\langle X, Y\rangle_E \ve - \langle X, \ve\rangle_E Y - \langle Y, \ve\rangle_E X),$$
where $\nabla^E$ denotes the Levi-Civita connection on $\mathbb{R}^{n+1}_+$ with respect to the standard Euclidean metric, $\langle\cdot,\cdot\rangle_E$ denotes the standard Euclidean inner product, and $\ve = \ve_{n+1}$.
\end{proposition}

Let $\{\vv_i\}_{i = 1}^n$ be a basis of $T_p\Sigma_t$, and denote the induced metric on $\Sigma_t$ by $$g_{ij} = \langle \vv_i,\vv_j\rangle_H.$$
 Denote the second fundamental form on $\Sigma_t$ by
$$a_{ij} = \langle \nabla^H_{\vv_i}\vv_j, \vnu_H\rangle_H,$$ 
so that the mean curvature of $\Sigma_t$ with respect to the hyperbolic metric is
$$H = g^{ij}a_{ij},$$
where $g^{ij}$ is the inverse of $g_{ij}$. With these we have
\begin{proposition}
$$\kappa_i^H = x_{n+1} \kappa_i^E + \vnu^{n+1}\,,$$
where $\kappa_i^H$ and $\kappa_i^E$ are hyperbolic and Euclidean principle curvatures of $\Sigma_t$, respectively, and $\vnu^{n+1} = \langle \vnu_E, \ve\rangle_E$. Therefore,
$$H = x_{n+1}H^E+ n \vnu^{n+1} \,,$$
where $H^E$ is the Euclidean mean curvature and $\vnu_E$ is the Euclidean unit normal of $\Sigma_t$. That is, $\vnu_H = x_{n+1}\vnu_E$.
\end{proposition}
\begin{proof} 
Note that the hyperbolic principle curvatures $\kappa_i^H$'s are the roots of 
\begin{align*}
\det \left( a_{ij} - \kappa^H g_{ij}\right) &= \det \left( \frac{a^E_{ij}}{x_{n+1}} - \frac{\vnu^{n+1}}{x_{n+1}^2} g_{ij}^E - \kappa^H \frac{g^E_{ij}}{x_{n+1}^2}\right) \\
&= x_{n+1}^{-n} \det \left( a_{ij}^E - \frac{\kappa^H -\vnu^{n+1}}{x_{n+1}} g_{ij}^E\right)\,,
\end{align*}
so that the proposition follows from
\begin{align*}
\kappa_i^E= \frac{1}{x_{n+1}} \left(\kappa_i^H - \vnu^{n+1}\right). 
\end{align*}
\end{proof}

\begin{proposition}\label{evolution-1}
 For a function $f : \Sigma_t \rightarrow \mathbb{R}$, where $\Sigma_t$ moves by \eqref{MMCF0}, we have
\begin{align*}
\left(\frac{\partial}{\partial t} - \Delta\right) f = & -x_{n+1}^2(\Delta_E f - \langle \nabla^E_{\vnu_E} \nabla^E f,\vnu_E\rangle_E)\\
& + x_{n+1}((n-2)\langle \nabla^E f, \ve\rangle_E  + 2\langle \nabla^E f, \vnu_E\rangle_E\langle\vnu_E,\ve\rangle_E - \sigma\langle \nabla^E f, \vnu_E\rangle_E),
\end{align*}
where $\Delta$ is the Laplace-Beltrami operator on $\Sigma_t$, $\frac{\partial}{\partial t} = F_*(\partial/\partial t) = (H - \sigma)\vnu_H$, $\Delta_E$ is the standard Euclidean Laplacian, and $\nabla^E f$ is the Euclidean gradient of $f$.
\end{proposition}

\begin{proof} Notice first 
\begin{align*}
\nabla f &=\, \nabla^H f - \langle \nabla^H f, \vnu_H\rangle_H\vnu_H\,,\\
\Div &=\, \Div_H - \langle\nabla^H_{\vnu_H}\cdot,\vnu_H\rangle_H\,,\\
\nabla^H f &=\, x_{n+1}^2\nabla^E f\,,\\
\Div_H &=\, \Div_E - \frac{n+1}{x_{n+1}}\langle \cdot, \ve\rangle_E\,.
\end{align*}
Along with Proposition \ref{prop2.1}, these give
\begin{align*}
\Delta f = &\Div\nabla f \\
= &\Div_H(\nabla^H f -\langle \nabla^H f, \vnu_H\rangle_H\vnu_H) -\langle\nabla^H_{\vnu_H}(\nabla^H f - \langle \nabla^H f, \vnu_H\rangle_H\vnu_H),\vnu_H\rangle_H\\
= & \Div_H\nabla^H f  -\langle \nabla^H f, \vnu_H\rangle_H\Div_H\vnu_H -\vnu_H \langle \nabla^H f, \vnu_H\rangle_H \\
&\- \langle\nabla^H_{\vnu_E} \nabla^H f,\vnu_E\rangle_E + \vnu_H \langle \nabla^H f, \vnu_H\rangle_H\\
= & \Div_H\nabla^H f - \langle\nabla^H_{\vnu_E} \nabla^H f,\vnu_E\rangle_E+ H\langle \nabla^H f, \vnu_H\rangle_H \\
= & \Div_E(x_{n+1}^2\nabla^E f) - (n+1)x_{n+1}\langle \nabla^E f, \ve\rangle_E  \\
&-\langle \nabla^E_{\vnu_E}(x_{n+1}^2\nabla^E f),\vnu_E\rangle_E -{x_{n+1}}\langle \vnu_E,\nabla^E f\rangle_E\langle \vnu_E,\ve\rangle_E \\
& + x_{n+1}\langle \vnu_E,\ve\rangle_E\langle \nabla^E f,\vnu_E\rangle_E  + x_{n+1}\langle \nabla^E f,\ve\rangle_E +H\langle \nabla^E f, \vnu_H\rangle\\
= & x_{n+1}^2\Div_E\nabla^E f + 2x_{n+1}\langle\nabla^E f, \ve\rangle_E -(n+1)x_{n+1}\langle \nabla^E f, \ve\rangle_E \\
& - x_{n+1}^2\langle \nabla^E_{\vnu_E}\nabla^E f,\vnu_E\rangle_E - 2x_{n+1}\langle \vnu_E,\ve\rangle_E \langle\nabla^E f,\vnu_E\rangle_E \\
&+ x_{n+1}\langle \nabla^E f,\ve\rangle_E +H\langle \nabla^E f, \vnu_H\rangle_E \\
= & x_{n+1}^2(\Delta_E f - \langle \nabla^E_{\vnu_E}\nabla^E f,\vnu_E\rangle_E) -x_{n+1}((n-2)\langle \nabla^E f, \ve\rangle_E\\
& - 2\langle\vnu_E,\ve\rangle_E\langle\nabla^E f,\vnu_E\rangle_E) +H\langle \nabla^E f, \vnu_H\rangle_E.
\end{align*}
Combining this with
$$\frac{\partial}{\partial t}f = (H - \sigma)\vnu_H f = H\langle \nabla^E f, \vnu_H\rangle_E - x_{n+1}\sigma\langle \nabla^E f, \vnu_E\rangle_E$$
gives the desired result.
\end{proof}

Now note that the Riemann curvature tensor is 
$$(R^H)_{\alpha\beta\gamma\delta} = \langle (R^H)(\ve_{\alpha},\ve_{\beta})\ve_\gamma,\ve_\delta\rangle_H = \delta_{\alpha\delta}\delta_{\beta\gamma}-\delta_{\alpha\gamma} \delta_{\beta\delta}\,,$$
since $\mathbb{H}^{n+1}$ has constant sectional curvature $-1$. In particular, $\nabla R^H = 0$. Also, the Gauss equation in this setting reads as
$$\text{Gauss: } {R}_{ijkl} = a_{ik}a_{jl} - a_{il}a_{jk}+ (R^H)_{ijkl},$$
where the index $0$ denotes the $\vnu_H$ direction. Note also that we have the interchange of two covariant derivatives on a two tensor:
$$\nabla_j\nabla_ia_{kl} = \nabla_i\nabla_ja_{kl} +a_{km} {R}^{\quad m}_{jil} +a_{lm} {R}^{\quad m}_{jik}\,,$$
where ${R}^{\quad m}_{ijk} = g^{ml}R_{ijkl}$. Using these equations one can derive the following well-known Simons' identity.

\begin{lemma} On $\Sigma_t \subset \mathbb{H}^{n+1}$, we have\\
(i) (Simons' identity)
$$\Delta  a_{ij} = \nabla_i\nabla_j H + Ha_{mi}a^m_j - |A|^2a_{ij} -na_{ij} + H\delta_{ij},$$
where $\Delta$ is the Laplacian for tensors on $\Sigma_t$, $\nabla$ the  covariant derivative on $\Sigma_t$, $\nabla_i = \nabla_{\vv_i}$ and $A = (a_{ij})$ the second fundamental form on $\Sigma_t$, all with respect to the induced hyperbolic metric.\\
(ii) \quad $\Delta |A|^2 = 2a^{ij}\nabla_i\nabla_j H + 2H \text{Tr}(A^3) - 2|A|^4 -2n|A|^2 + 2H^2 + 2|\nabla A|^2.$
\end{lemma}
\begin{proof} We include a proof for the sake of completeness. See also \cite{H86} for general ambient manifolds. Fix a point on $\Sigma_t$. We will work on a normal coordinate at this point. For (i), we have
\begin{align*}
\Delta a_{ij} &  = \,\nabla_k\nabla_k a_{ij} = \,\nabla_k\nabla_j a_{ik}  \\
& = \,\nabla_i\nabla_ka_{jk} + a_{jl}{R}^{\quad l}_{ kik} + a_{kl}{R}^{\quad l}_{kij}\\
& = \,\nabla_i\nabla_jH + a_{j}^l(a_{kk}a_{il}-a_{kl}a_{ik}+(R^H)_{kik}^{\quad l}) + a_{kl}(a_{kj}a_{il}-a_{kl}a_{ij}+(R^H)_{kij}^{\quad l})\\
& = \,\nabla_i\nabla_j H + Ha_{il}a^l_j + a_{jl}(\delta_{kl}\delta_{ik}-\delta_{kk}\delta_{il})- |A|^2a_{ij} + a_{kl}(\delta_{kl}\delta_{ij}-\delta_{jk}\delta_{il})\\
& = \,\nabla_i\nabla_j H +Ha_{il}a^l_j- |A|^2a_{ij}  -na_{ij} + H\delta_{ij}.
\end{align*}
For (ii), we have
\begin{align*}
\Delta |A|^2 &  = \,2a^{ij}\Delta a_{ij} + 2|\nabla  A|^2\\
& =  \,2a^{ij}\nabla_i\nabla_j H + 2H\text{Tr}(A^3) - 2|A|^4 -2n|A|^2 + 2H^2 + 2|\nabla A|^2\,.\qedhere
\end{align*}
\end{proof}

In order to obtain the estimates on higher order derivatives, we also need the evolution equation for the second fundamental forms.
\begin{lemma} \label{evolution-ofA}
On $\Sigma_t \subset \mathbb{H}^{n+1}$, we have
\begin{align*}
&(i)  \quad \frac{\partial}{\partial t}  a_{ij} \,=\, \nabla_i\nabla_j H -(H - \sigma)a^k_ia_{jk} + (H-\sigma)(R^H)_{i0j0}\,,\\
&(ii) \quad \frac{\partial}{\partial t}   |A|^2 \,= \, 2a^{ij}\nabla_i\nabla_jH +2(H-\sigma)\text{Tr}(A^3)-2H(H-\sigma)\,,\\
&(iii)\quad \left(\frac{\partial}{\partial t} -\Delta\right)  |A|^2 \,= \, 2|A|^4 + 2n|A|^2 -2|\nabla A|^2-4H^2+2\sigma(H - \text{Tr} (A^3) )\,.
\end{align*}

\end{lemma}
\begin{proof}

(i) The evolution equation for $a_{ij}$ along the mean curvature flow in general Riemannian manifold can be found in \cite{H86}. Here, for completeness, we prove it in our setting. Note that $\nabla^H_{\vv_i}\vv_j = a_{ij}\vnu_H$, we compute
\begin{align*}
\frac{\partial}{\partial t} a_{ij} 
&=\langle \nabla^H_{\frac{\partial}{\partial t}}\nabla^H_{\vv_i}\vv_j,\vnu_H\rangle_H+\langle\nabla^H_{\vv_i}{\vv_j}, \nabla^H_{\frac{\partial}{\partial t}}{\vnu}\rangle_H \\
&=\langle \nabla^H_{\vv_i}\nabla^H_{\vv_j}\frac{\partial}{\partial t},\vnu_H\rangle_H+\langle (R^H)(\vv_i, {\partial}/{\partial t})\vv_j,\vnu_H\rangle_H+\langle\nabla^H_{\vv_i}{\vv_j}, -\nabla{H}\rangle_H \\
& = \, \langle \nabla^H_{\vv_i}\nabla^H_{\vv_j}((H-\sigma)\vnu_H),\vnu_H\rangle_H + (H-\sigma)(R^H)_{i0j0}-\Gamma^k_{ij}\vv_k(H)\\
& = \,\langle \nabla^H_{\vv_i}(\nabla^H_{\vv_j}H\vnu_H)- \nabla^H_{\vv_i}((H-\sigma)a_j^k\vv_k),\vnu_H\rangle_H + (H-\sigma)(R^H)_{i0j0}-\Gamma^k_{ij}\vv_k(H)\\
& = \nabla_i^H(\nabla_j^H H)-(H - \sigma)a^k_ia_{jk} + (H-\sigma)(R^H)_{i0j0}-\Gamma^k_{ij}\vv_k(H)\\
& = \vv_i(\vv_j(H))-(H - \sigma)a^k_ia_{jk} + (H-\sigma)(R^H)_{i0j0}-\Gamma^k_{ij}\vv_k(H)\\
& =\nabla_i\nabla_j H -(H - \sigma)a^k_ia_{jk} + (H-\sigma)(R^H)_{i0j0},
\end{align*}
where $\nabla{H}=g^{rs}\vv_r(H)\vv_s$. Suppose $\{x_i\}$ is a local coordinate on $\mathbb{S}^n_+$, then $\vv_i=F_*(\frac{\partial}{\partial x_i})$ and $\vv_i(H)=\frac{\partial H}{\partial x_i}$, $ \vv_i(\vv_j(H))=\frac{\partial^2H}{\partial x_i\partial x_j}$.

(ii) Notice $\frac{\partial}{\partial t} g^{ij} = 2(H-\sigma)g^{ik}g^{jl}a_{kl}$, so that
\begin{align*}
\frac{\partial}{\partial t} |A|^2 =& \, \frac{\partial}{\partial t} \left(g^{ij}g^{kl}a_{ik}a_{jl}\right)\\
 =& \,4(H-\sigma)a^{ij}a_{ik}a^k_j+2a^{ij}\left(\nabla_i\nabla_jH - (H-\sigma)a_{ik}a^k_j + (H-\sigma)(R^H)_{i0j0}\right)\\
 =&  \, 2a^{ij}\nabla_i\nabla_jH +2(H-\sigma)\text{Tr}(A^3)-2H(H-\sigma).
\end{align*}

(iii) Combining (ii) with the Simons' identity.
\end{proof}

Finally, we note that there is a $C^0$-estimate that comes for free.
\begin{remark} \label{remark1} Notice $|x|_E$ is bounded above on any compact region of $\Sigma_t$, by the same constant, for all time. To see this, there exist, for any $r > 0$, caps $\{ (x_1,\dots,x_{n+1}) \in \mathbb{H}^{n+1}  : (x_1)^2 + \cdots (x_n)^2 + (x_{n+1} +\sigma r/n)^2 = r^2\}$, with constant hyperbolic mean curvature $\sigma$. These caps have bounded $|x|_E$. The result follows from a comparison principle for MMCF. That is, the ratio of the Euclidean radial height above a fixed point in $\partial_\infty\mathbb{H}^{n+1}$ between two hypersurfaces (with one compact) moving by MMCF in hyperbolic space is non-decreasing in time.
\end{remark}

\section{Interior gradient estimates}\label{interior}

The MMCF \eqref{MMCF0} for complete radial graphs is a (degenerate) quasi-linear parabolic PDE, see \eqref{MMCF1}. We would like to use the conventional maximum principle techniques to obtain interior estimates. Similar interior estimates were obtained in \cite[Section 9]{LX12} using the same techniques. However, the estimate there is not uniform in $\epsilon$ and therefore it is not sufficient in our current case. In order to overcome the degeneracy at infinity of the PDE and achieve the uniform interior estimate, we first need to find an appropriate space-time cut-off function. To do so, we let $r(x)$ be the hyperbolic distance from a point $x \in\mathbb{H}^{n+1}$ to the $x_{n+1}$-axis. Then 
$$\cosh r = \frac{|x|_E}{x_{n+1}}\,,$$ where $|x|_E = \sqrt{\langle x,x\rangle_E}$, see e.g. \cite[Cor. A.5.8.]{BP}. In the following, we let $\mathbf{z} = \frac{x}{|x|_E}$.

\begin{proposition}\label{prop2}
$$\left(\frac{\partial}{\partial t} - \Delta \right)\cosh r= \frac{1}{\cosh r}(1 - \langle \vnu_E, \mathbf{z}\rangle_E^2) - (n - \sigma\langle \vnu_E, \ve\rangle_E)\cosh r - \sigma\langle \vnu_E, \mathbf{z}\rangle_E.$$ 
\end{proposition}

\begin{proof} Notice 
\begin{align*}
\nabla^E |x|_E \,&=\, \mathbf{z}\,,\\
\nabla^E_{\vnu_E}\nabla^E |x|_E \,&=\, \nabla^E_{\vnu_E} \mathbf{z} \,=\, \vnu_E |x|_E^{-1}x + |x|_E^{-1}\vnu_E \,=\, -|x|_E^{-1}\langle \mathbf{z}, \vnu_E\rangle_E \mathbf{z} +  |x|_E^{-1}\vnu_E\,,\\
\Delta_E |x|_E \,&=\, \Div_E\mathbf{z} = -|x|_E^{-1} +|x|_E^{-1}(n+1) = n|x|_E^{-1}\,.
\end{align*}
Moreover, we have
\begin{align*}
\nabla^E x_{n+1}^{-1} \,&=\, -x_{n+1}^{-2}\ve\,,\\
\nabla^E_{\vnu_E}\nabla^E x_{n+1}^{-1} \,&=\, 2x_{n+1}^{-3}\langle \ve,\vnu_E\rangle_E \ve\,,\\
\Delta_E x_{n+1}^{-1} \,&=\, 2x_{n+1}^{-3}\,,\\
\nabla^E \cosh r \,&= \,x_{n+1}^{-1}\mathbf{z} -x_{n+1}^{-2}|x|_E\ve \,=\, x_{n+1}^{-1} \mathbf{z} - x_{n+1}^{-1} (\cosh r) \ve\,,\\
x_{n+1}\nabla^E \cosh r \,&=\, \mathbf{z} - (\cosh r)\ve\,,
\end{align*}
and
\begin{align*}
&\nabla^E_{\vnu_E} \nabla^E \cosh r \\
=&\, \nabla^E_{\vnu_E}(x_{n+1}^{-1} \mathbf{z} - x_{n+1}^{-1} (\cosh r) \ve)\\
= & \,-x_{n+1}^{-2}\langle \vnu_E, \ve\rangle_E \mathbf{z} + x_{n+1}^{-1}( -|x|_E^{-1}\langle \mathbf{z}, \vnu_E\rangle_E \mathbf{z} +  |x|_E^{-1}\vnu_E) + x_{n+1}^{-2}\langle \vnu_E, \ve\rangle_E(\cosh r)\ve \\
& -x_{n+1}^{-1}\langle x_{n+1}^{-1} \mathbf{z} - x_{n+1}^{-1} (\cosh r) \ve, \vnu_E\rangle_E\\
= & \,x_{n+1}^{-2}\bigg(-\langle\ve,\vnu_E\rangle_E\mathbf{z} -\frac{1}{\cosh r}\langle \mathbf{z}, \vnu_E\rangle_E\mathbf{z} + \frac{1}{\cosh r}\vnu_E -\langle \mathbf{z},\vnu_E\rangle_E\ve + 2\cosh r\langle \ve,\vnu_E\rangle_E\ve\bigg)\,.
\end{align*}

Now, since $\langle \mathbf{z} , \ve\rangle_E = \frac{1}{\cosh r}$, we have
\begin{align*}
\Delta_E\cosh r =& \, \Delta_E x_{n+1}^{-1}|x|_E\\
= & \, 2\langle \nabla^E x_{n+1}^{-1}, \nabla^E |x|_E \rangle_E +x_{n+1}^{-1}\Delta_E |x|_E + |x|_E\Delta_Ex_{n+1}^{-1}\\
= &\, x_{n+1}^{-2}\left((n-2)\frac{1}{\cosh r}  +2\cosh r\right)\,.
\end{align*}

Therefore, we finally arrive at
\begin{align*}
\left(\frac{\partial}{\partial t} - \Delta\right) \cosh r = &\, -x_{n+1}^2(\Delta_E\cosh r - \langle \nabla^E_{\vnu_E} \nabla^E\cosh r,\vnu_E\rangle_E) \\
& + x_{n+1}[(n-2)\langle \nabla^E\cosh r, \ve\rangle_E  + 2\langle \nabla^E\cosh r, \vnu_E\rangle_E\langle \ve, \vnu_E\rangle_E\\
& - \sigma\langle \nabla^E\cosh r, \vnu_E\rangle_E] \\
= &\, (2-n)\langle \mathbf{z}, \ve\rangle_E  -2\cosh r-\frac{1}{\cosh r}\langle \mathbf{z}, \vnu_E\rangle_E^2+ \frac{1}{\cosh r}\\
& -2\langle \mathbf{z},\vnu_E\rangle_E \langle \ve,\vnu_E\rangle_E  + 2\cosh r\langle \ve,\vnu_E\rangle_E^2 \\
& + (n-2)\langle \mathbf{z},\ve\rangle_E -(n-2)\cosh r + 2\langle \mathbf{z},\vnu_E\rangle_E\langle \ve,\vnu_E\rangle_E \\
& - 2\cosh r\langle \ve,\vnu_E\rangle_E^2 - \sigma\langle \mathbf{z},\vnu_E\rangle_E + \sigma \cosh r\langle \ve,\vnu_E\rangle_E \\
= &\, \frac{1}{\cosh r}(1 - \langle \vnu_E,\mathbf{z}\rangle_E^2) - (n - \sigma\langle \ve, \vnu_E \rangle_E)\cosh r- \sigma\langle \mathbf{z},\vnu_E\rangle_E. \qedhere
\end{align*}
\end{proof}

\vskip 2mm
Now, for any $R > 0$, we define a space-time cut-off function (c.f. \cite{U03})
$$\eta = \cosh R - e^{(n+\sigma)t}\left(\cosh r + \frac{\sigma}{n+\sigma}\right)\,.$$ 
Then, for $\sigma \geq 0$ we have
\begin{align*}
\left(\frac{\partial}{\partial t} - \Delta\right)\eta  =&\,  -e^{(n+\sigma)t}\left((n+\sigma)\cosh r +\sigma + \left(\frac{\partial}{\partial t} - \Delta\right)\cosh r\right) \\
=&\, -e^{(n+\sigma)t}\bigg[(n+\sigma)\cosh r +\sigma +  \frac{1}{\cosh r}(1 - \langle \vnu_E,\mathbf{z}\rangle_E^2)\\
& - (n - \sigma\langle \ve, \vnu_E\rangle_E)\cosh r - \sigma\langle \mathbf{z},\vnu_E\rangle_E\bigg] \\
= &\,  -e^{(n+\sigma)t}\bigg[ \frac{1}{\cosh r}(1 - \langle \vnu_E,\mathbf{z}\rangle_E^2)  + \sigma(1-\langle \mathbf{z},\vnu_E\rangle_E \\
& + \cosh r(1 + \langle \ve, \vnu_E\rangle_E))\bigg] \leq \,0 \,.
\end{align*}

\begin{remark}
We will only deal with the case of $\sigma \geq 0$. The case of $\sigma <0$ can be handled using the hyperbolic isometric reflection $x^\ast = \frac{x}{|x|_E^2}$ w.r.t. $\mathbb{S}_+^n$.
\end{remark}

\begin{remark}
Notice that 
$$
\vnu_E = \frac{\mathbf{z} - \nabla v}{\sqrt{1+|\nabla v|^2}} \quad \text{and}\quad \left\langle \vnu_E, \mathbf{z}\right\rangle_E = \frac{1}{|x|_E} \left\langle \vnu_E, x\right\rangle_E = \frac{1}{\sqrt{1+|\nabla v|^2}}\,.
$$
Therefore, in order to get the interior gradient estimate on $|\nabla v|$, we will need to get a positive lower bound on $\left\langle \vnu_E, \mathbf{z}\right\rangle_E$, which is (almost) equivalent to $\left\langle \vnu_E, x\right\rangle_E = x_{n+1} \left\langle \vnu_H, x\right\rangle_H$, thanks to the $C^0$-estimate on $|x|_E$ using appropriate barriers (see Remark \ref{remark1}). Thus, in the following we will first look at the evolution equation of $\left\langle \vnu_H, x\right\rangle_H$ and finally arrive at the evolution equation of $\left\langle \vnu_E, x\right\rangle_E$ (see Proposition \ref{prop1}). Then the cut-off function and maximum principle techniques apply conventionally.
\end{remark}

From here on suppose the $\vv_i$'s are in fact a normal coordinate basis of  $T_p\Sigma_t$ with respect to the hyperbolic metric. We may extend the vector fields $\vv_i$ and $\vnu_H$ on $\Sigma_t$ to a neighborhood of $\mathbb{H}^{n+1}$ by requiring that $\vv_i$ is constant along the integral curves of $x$, so that $[\vv_i,x]  = [\vnu_H,x] = 0$, where, e.g., $[\vv_i,x]$ is the Lie bracket of $\vv_i$ and $x$. See, e.g., \cite{Bar84}. Note that the Codazzi equation becomes, since $\mathbb{H}^{n+1}$  has constant sectional curvature,
\begin{equation}\label{codazzi}
a_{ij, k} = a_{ik,j}.
\end{equation}

\begin{proposition} For radial graphs moving by MMCF,
$$\left(\frac{\partial}{\partial t} - \Delta\right)\langle \vnu_H,x\rangle_H = (|A|^2 - n)\langle \vnu_H,x\rangle_H,$$
where $|A|^2 = g^{ij}g^{kl}a_{ik}a_{jl}$ is the norm squared of the second fundamental form on $\Sigma_t$. 
\end{proposition}

\begin{proof} We have, using $[\vv_i,x] = 0$, \eqref{Ric comp}, and Codazzi equation \eqref{codazzi}, and summing over repeated indices,
\begin{align*}
\Delta \langle \vnu_H,x\rangle_H = &\, \vv_i\vv_i \langle \vnu_H,x\rangle_H 
= \,\vv_i\langle \nabla^H_{\vv_i}\vnu_H, x\rangle_H 
+ \vv_i\langle \vnu_H,\nabla^H_{\vv_i} x\rangle_H \\
= &\, -\langle \nabla^H_{\vv_i}a_{ij}\vv_j, x\rangle_H -|A|^2\langle \vnu_H,x\rangle_H
- 2\langle a_{ij}\vv_j,  \nabla^H_{\vv_i}x\rangle_H \\
&+\langle \vnu_H, (R^H)(x,\vv_i)\vv_i\rangle_H + \langle\vnu_H, \nabla^H_x \nabla^H_{\vv_i}\vv_i\rangle_H \\
= & \, -\vv_j(H)
\langle \vv_j,x\rangle_H +\langle  (R^H)(x,\vv_i)\vv_i,\vnu_H\rangle_H -|A|^2\langle \vnu_H,x\rangle_H +a_{ij}xg^{ij} + xa_{ii} \\
= &  \,-\langle \nabla H, x\rangle_H -\Ric^H(\vnu_H,\vnu_H)\langle \vnu_H,x\rangle_H -|A|^2\langle \vnu_H,x\rangle_H + x(H)\\
= & \, (n-|A|^2)\langle \vnu_H,x\rangle_H -\langle \nabla  H, x\rangle_H  + x(H)\,.
\end{align*}
Notice $\nabla^H_{\frac{\partial}{\partial t} }\vnu_H$ is tangential, and $[ \frac{\partial}{\partial t},\vv_i] = 0$ from the  naturality of the Lie bracket. So, 
$$\langle \nabla^H_{\frac{\partial}{\partial t} }\vnu_H, \vv_i\rangle_H  = -\langle \vnu_H, \nabla^H_{\vv_i } \frac{\partial}{\partial t} \rangle_H = -\vv_i(H-\sigma) -(H-\sigma)\langle\vnu_H,\nabla^H_{\vv_i}\vnu_H\rangle_H =-\vv_iH,$$
which implies
$$\nabla^H_{\frac{\partial}{\partial t} }\vnu_H = -\nabla  H.$$
Also,
$$\langle \vnu_H,\nabla^H_{\vnu_H } x\rangle_H = \langle \vnu_E, \nabla^E_{\vnu_E} x + \frac{1}{x_{n+1}}(\langle \vnu_E,x\rangle_E \,\ve - \langle \vnu_E,\ve\rangle_E x-\langle x,\ve\rangle_E \vnu_E\rangle_E = 0$$
since $\nabla^E_{\vnu_E} x = \vnu_E$ and $\langle x,\ve\rangle_E = x_{n+1}$.
Hence,
\begin{align*}
\frac{\partial}{\partial t} \langle \vnu_H,x\rangle_H 
& = \,\langle \nabla^H_{\frac{\partial}{\partial t} }\vnu_H, x\rangle_H +(H- \sigma)\langle \vnu_H,\nabla^H_{\vnu_H } x\rangle_H \\
&= \,-\langle \nabla H, x\rangle_H.
\end{align*}
Finally, notice that $x(H) = 0$ since $x$ is a Killing vector field in $\mathbb{H}^{n+1}$, c.f. \cite[Appendix]{HLZ}.
\end{proof}

\begin{proposition} \label{prop1} For radial graphs moving by MMCF,
\begin{equation}\label{evosupp}
\left(\frac{\partial}{\partial t} - \Delta\right)\langle \vnu_E,x\rangle_E= (|A|^2 - \sigma\langle \vnu_E, \ve\rangle_E)\langle \vnu_E,x\rangle_E - 2\langle \nabla\langle \vnu_E,x\rangle_E,x_{n+1}\ve\rangle_H.
\end{equation}
\end{proposition}

\begin{remark}\label{obstruction}
In the case of MCF, i.e., $\sigma=0$, equation \eqref{evosupp} and the maximum principle yield immediately a global gradient bound for the approximate MCF (starting from the compact hypersurface $\Sigma_0^\epsilon$), which ensures the global existence of the approximate MCF, see \cite{U03}. On the other hand, in the case $\sigma \neq 0$, the maximum principle is not applicable directly, but thanks to the existence result from \cite{LX12} for the approximate MMCF we are able to get around with this, see Section \ref{proofofmain}.
\end{remark}
\begin{proof} We have, using $\nabla x_{n+1} = \nabla^Hx_{n+1} -\langle \nabla^H x_{n+1},\vnu_H\rangle_H\vnu_H = x_{n+1}^2(\ve - \langle \vnu_E,\ve\rangle_E \vnu_E)$, that 
$$|\nabla x_{n+1}|_H^2 = x_{n+1}^2(1-\langle \vnu_E,\ve\rangle_E^2).$$
Hence, using Proposition \ref{evolution-1}, we have
\begin{align*}
\left(\frac{\partial}{\partial t} -\Delta\right)\langle \vnu_E,x\rangle_E  = &\, \left(\frac{\partial}{\partial t} -\Delta\right)\left(x_{n+1}\langle \vnu_H,x\rangle_H\right) \\
= &\, x_{n+1}\left(\frac{\partial}{\partial t} -\Delta\right)\langle \vnu_H,x\rangle_H +\langle \vnu_H,x\rangle_H \left(\frac{\partial}{\partial t} -\Delta\right) x_{n+1}\\
& - 2\langle \nabla x_{n+1},\nabla\langle\vnu_H,x\rangle_H\rangle_H \\
=&\, (|A|^2-n)\langle \vnu_E,x\rangle_E + \langle\vnu_E,x\rangle_E(n-2 + 2\langle \vnu_E,\ve\rangle_E^2 - \sigma\langle \vnu_E,\ve\rangle_E)\\
& -2\left\langle \nabla x_{n+1}, \frac{1}{x_{n+1}}\nabla \langle\vnu_E,x\rangle_E\right\rangle_H 
-2\left\langle \nabla  x_{n+1},\langle\vnu_E,x\rangle_E\nabla  \frac{1}{x_{n+1}}\right\rangle_H  \\
= &\, (|A|^2 -2 + 2\langle \vnu_E,\ve\rangle_E^2 - \sigma\langle \vnu_E,\ve\rangle_E)\langle \vnu_E,x\rangle_E \\
& -2\left\langle x_{n+1}\ve, \nabla \langle\vnu_E,x\rangle_E\right\rangle_H +2\langle\vnu_E,x\rangle_E(1-\langle \vnu_E,\ve\rangle_E^2)  \\
=& \, (|A|^2 - \sigma\langle \vnu_E,\ve\rangle_E)\langle \vnu_E,x\rangle_E - 2\langle \nabla \langle \vnu_E,x\rangle_E,x_{n+1}\ve\rangle_H. \qedhere
\end{align*}
\end{proof}

Now, in order to obtain the interior estimate using maximum principle techniques, we multiply $\langle \vnu_E,x\rangle_E^{-1}$ by the space-time cut-off function and let
\begin{equation}
\xi = \eta^3\langle \vnu_E,x\rangle_E^{-1} = \left(\cosh R - e^{(n+\sigma)t}\left(\cosh r + \frac{\sigma}{n+\sigma}\right)\right)^3\langle \vnu_E,x\rangle_E^{-1} .
\end{equation}

\begin{proposition}
For radial graphs moving by MMCF with $\sigma \in  [0,n)$,
$$\left(\frac{\partial}{\partial t} - \Delta \right)\xi \le (n+2)\xi.$$
\end{proposition}

\begin{proof}  This is a straight-forward calculation.
\begin{align*}
\left(\frac{\partial}{\partial t} -\Delta \right)\xi  
=&  \, \langle\vnu_E,x\rangle_E^{-1}\left(\frac{\partial}{\partial t} -\Delta \right)\eta^3 +\eta^3\left(\frac{\partial}{\partial t} -\Delta \right)\langle\vnu_E,x\rangle_E^{-1} -2\left\langle \nabla \eta^3,\nabla \langle\vnu_E,x\rangle_E^{-1}\right\rangle_H \\
=& \, 3\eta^2\langle\vnu_E,x\rangle_E^{-1}\left(\frac{\partial}{\partial t} -\Delta \right)\eta - 6\eta\langle\vnu_E,x\rangle_E^{-1}|\nabla \eta|_H^2
-\eta^3\langle \vnu_E,x\rangle_E^{-2}\left(\frac{\partial}{\partial t} -\Delta \right)\langle\vnu_E,x\rangle_E\\
 &- 2\eta^3\langle\vnu_E,x\rangle_E^{-3}|\nabla \langle \vnu_E,x\rangle_E|^2_H
+6\eta^2\langle\vnu_E,x\rangle_E^{-2}\left\langle \nabla \eta,\nabla \langle\vnu_E,x\rangle_E\right\rangle_H \\
\leq& \, -\eta^3\langle \vnu_E,x\rangle_E^{-2}\left((|A|^2 - \sigma\langle \vnu_E, \ve\rangle_E)\langle \vnu_E,x\rangle_E  - 2\langle \nabla \langle \vnu_E,x\rangle_E,x_{n+1}\ve\rangle_H\right)\\
&- \frac{1}{2}\eta^3\langle\vnu_E,x\rangle_E^{-3}|\nabla \langle \vnu_E,x\rangle_E|^2_H\\
\leq& \,\eta^3\langle\vnu_E,x\rangle_E^{-1}\left(\langle\vnu_E,\ve\rangle_E\sigma-|A|^2 + 2\right) \,\leq\, (n+2)\xi\,,
\end{align*}
where we have used
$$2\eta^3\langle \vnu_E,x\rangle_E^{-2}\langle\nabla \langle \vnu_E,x\rangle_E,x_{n+1}\ve\rangle_H \le \frac{1}{2}\eta^3\langle \vnu_E,x\rangle_E^{-3}|\nabla \langle\vnu_E,x\rangle_E|_H^2 + 2\eta^3\langle\vnu_E,x\rangle_E^{-1}\,,$$
and 
$$6\eta^2\langle\vnu_E,x\rangle_E^{-2}\left\langle \nabla \eta,\nabla \langle\vnu_E,x\rangle_E\right\rangle_H \le 6\eta\langle\vnu_E,x\rangle_E^{-1}|\nabla \eta|_H^2 + \frac{3}{2} \eta^3\langle\vnu_E,x\rangle_E^{-3}|\nabla \langle \vnu_E,x\rangle_E|^2_H\,,$$
from Young's inequality.
\end{proof}

The following theorem is the main technical interior gradient estimate.
\begin{theorem} \label{intgrad}
For any $R \geq \cosh^{-1}\left(\frac{\sigma}{n+\sigma}e^{(n+\sigma)T}\right)$ and $\theta \in \left(\frac{\sigma}{(n+\sigma)\cosh R} e^{(n+\sigma)T},1\right)$ such that $\{x \in \Sigma_t \mid r \le R\}$ is a compact radial graph for all $t\in [0,T]$, we have
\begin{equation*}
\sup_{\{x \in \Sigma_t \mid e^{(n+\sigma)t}(\cosh r + \frac{\sigma}{n+\sigma})\le  \theta\cosh R\}}\langle \vnu_E,\mathbf{z}\rangle_E^{-1} \le e^{(n+2)T + v_{\text{osc}}}(1-\theta)^{-3}\sup_{\{x\in \Sigma_0\mid r \le R\}}\langle \vnu_E,\mathbf{z}\rangle_E^{-1} \,,
\end{equation*}
where $v_{\text{osc}} = \max_{\{x\in \Sigma_t | r \le R\}\times [0,T]} v - \min_{\{x\in \Sigma_t | r \le R\}\times [0,T]}  v$ is the oscillation of the radial height of $x$ (see \eqref{radialheight}) in $\{x\in \Sigma_t\mid r \le R\}\times [0,T]$.
\end{theorem}

\begin{proof} The previous proposition and Hamilton's trick imply, for almost all $t\in (0,T)$,
$$\frac{d}{dt}\sup_{\{x\in \Sigma_t\mid r \le R\}} \xi \le (n+2)\sup_{\{x\in \Sigma_t\mid r \le R\}}\xi,$$
so we may integrate from $0$ to $T$ to obtain
$$\sup_{\{x\in \Sigma_T\mid r \le R\}} \eta^3\langle \vnu_E,x\rangle_E^{-1} \le e^{(n+2)T}\sup_{\{x\in \Sigma_0\mid r \le R\}}\eta^3\langle \vnu_E,x\rangle_E^{-1}.$$
Now notice $e^{v_{\text{min}}} \le |x|_E$ implies
$$e^{(n+2)T-v_{\text{min}}}\sup_{\{x\in \Sigma_0\mid r \le R\}}\eta^3\langle \vnu_E,\mathbf{z}\rangle_E^{-1} \ge e^{(n+2)T}\sup_{\{x\in \Sigma_0\mid r \le R\}}\eta^3\langle \vnu_E,x\rangle_E^{-1}.$$
Similarly, $e^{v_{\text{max}}} \ge |x|_E$ implies
$$e^{-v_{\text{max}}}\sup_{\{x\in \Sigma_T\mid r \le R\}}\eta^3\langle \vnu_E,\mathbf{z}\rangle_E^{-1} \le \sup_{\{x\in \Sigma_T\mid r \le R\}}\eta^3\langle \vnu_E,x\rangle_E^{-1}.$$
These two inequalities imply then
$$\sup_{\{x\in \Sigma_T\mid r \le R\}}\eta^3\langle \vnu_E,\mathbf{z}\rangle_E^{-1} \le e^{(n+2)T+v_{\text{max}}-v_{\text{min}}}\sup_{\{x\in \Sigma_0\mid r \le R\}}\eta^3\langle \vnu_E,\mathbf{z}\rangle_E^{-1}.$$
We also have
$$\sup_{\{x \in \Sigma_T \mid e^{(n+\sigma)t}(\cosh r + \frac{\sigma}{n+\sigma})\le  \theta\cosh R\}}\eta^3\langle \vnu_E,\mathbf{z}\rangle_E^{-1}\le \sup_{\{x\in \Sigma_T\mid r \le R\}}\eta^3\langle \vnu_E,\mathbf{z}\rangle_E^{-1},$$
and $\eta^3 \ge (1-\theta)^3\cosh^3R$ in ${\{x \in \Sigma_t \mid e^{(n+\sigma)t}(\cosh r + \frac{\sigma}{n+\sigma})\le  \theta\cosh R\}}$ since $\theta\cosh R + \eta \ge \cosh R$ there. We also have $\eta^3 \le \cosh^3 R$ everywhere. These facts, along with replacing $T$ with any $t \in [0,T)$, imply the result.
\end{proof}

\section{Interior estimates on higher order derivatives}\label{higherderiv}
\subsection{Estimates on the second derivatives}
Now let $u = \langle \vnu_E, x\rangle_E^{-1}$ and define 
$$\VP = \VP(u^2) = \frac{u^2}{1-ku^2}$$
where
$$k = \left(2\sup_{t\in [0,T]}\sup_{\{x\in \Sigma_t | r \le R\} } u^2\right)^{-1}\,.$$
Let $\VP^\prime$ denote differentiation of $\VP$ with respect to $u^2$. From Remark \ref{remark1},  we know that 
$$c_0 \le |x|_E^{-2} \le \VP$$
for some constant $c_0$ depending on $\Sigma_0$.

Combining Proposition \ref{prop1} with (iii) of Lemma \ref{evolution-ofA}, we obtain:
\begin{lemma}
On $\{x\in \Sigma_t | r \le R\}$ and $\Sigma_t$ moves by MMCF, we have
\begin{align*}
\left(\frac{\partial}{\partial t} -\Delta \right)  \left(|A|^2\VP\right)
\le &-k|A|^4\VP^2 
+ \left(\frac{c(n,c_0)}{k} 
- k\VP^\prime|\nabla v|^2\right)|A|^2\VP \\
& -\VP^{-1}\langle \nabla \VP,\nabla (|A|^2\VP)\rangle_H +\sigma^2 \VP\,.
\end{align*}
\end{lemma}

\begin{proof} We have
\begin{align*}
&\left(\frac{\partial}{\partial t} -\Delta \right) \left(|A|^2\VP\right) \\
=& \, \VP\left(\frac{\partial}{\partial t} -\Delta \right)  |A|^2 
+ |A|^2\left(\frac{\partial}{\partial t} -\Delta \right)  \VP - 2\langle \nabla |A|^2,\nabla \VP\rangle_H \\
:=& \, \text{I} + \text{II} + \text{III}\,.
\end{align*}

By (iii) of Lemma \ref{evolution-ofA}, we have
\begin{align*}
\text{I} &=\, \VP \left(2|A|^4 + 2n|A|^2 -2|\nabla A|^2-4H^2+2\sigma(H - \text{Tr} (A^3) )\right)\\
&
\le \VP\left(2|A|^4 
+ 2n|A|^2 
- 2|\nabla  A|^2
 - 4H^2 
+ \sigma\left(H^2c_2 
+ \frac{1}{c_2} 
+ \frac{|A|^2}{c_1} 
+ c_1|A|^4\right)\right)
\\
&
 \le \VP(2 + c_1\sigma)|A|^4 
+\VP\left( 2n+\frac{\sigma}{c_1}\right)|A|^2
 - 2\VP|\nabla A|^2
 + \frac{\sigma}{c_2}\VP
\end{align*}
where we used Young's inequality and the fact that $|\text{Tr}(A^3)| \le |A|^3$. We also chose constants $c_1, c_2$ such that $c_1\sigma \le c_0 k$ and $c_2\sigma \le 4$, where $c_0 \le \VP$.

For the second term II, by Proposition \ref{prop1} we have
\begin{align*}
\left(\frac{\partial}{\partial t} -\Delta \right)\VP
 &
= -2\VP^\prime u^3\left(\frac{\partial}{\partial t} -\Delta \right)\langle \vnu_E,x\rangle_E
 - 6\VP^\prime|\nabla u|^2
- 4\VP^{\prime\prime}u^2|\nabla u|^2
\\
& 
= -2\VP^\prime u^2(|A|^2-\sigma\langle \vnu_E,\ve\rangle_E) 
- 4\VP^\prime u\langle \nabla u,x_{n+1} \ve\rangle_H
- (6 + 8k\VP)\VP^\prime|\nabla u|^2
\end{align*}
since  $\VP^{\prime\prime}u^2= 2k\VP\VP^\prime$.

Therefore, using Young's inequality again we get
\begin{align*}
\text{II} 
\le -2u^2\VP^\prime |A|^4 
&
- (6 + 8k\VP)\VP^\prime |A|^2 |\nabla u|^2
+k\VP\VP^\prime|A|^2|\nabla u|^2 
+ \frac{4}{c_0k}|A|^2\VP+4n|A|^2\VP\,,
\end{align*}
since $\sigma < n, \VP^\prime u^2 \le 2 \VP$ and $\frac{\VP}{c_0} \ge 1$.

For the third term III, we compute:
\begin{align*}
\text{III} 
&
=\, -\VP^{-1}\langle \nabla \VP,\nabla ( |A|^2\VP) \rangle_H 
+ \VP^{-1} |A|^2|\nabla \VP|^2 - 
\langle \nabla |A|^2,  \nabla \VP\rangle_H 
\\
&
 = \,-\VP^{-1}\langle \nabla \VP,\nabla (|A|^2\VP)\rangle_H 
+4 \VP^{-1}(\VP^\prime u)^2|A|^2|\nabla u|^2 - 
4\VP^\prime u |A| \langle \nabla |A|, \nabla  u \rangle_H 
\\
&
\leq  \,-\VP^{-1}\langle \nabla \VP,\nabla (|A|^2\VP)\rangle_H 
+ 6 \VP^{-1}(\VP^\prime u)^2|A|^2|\nabla u|^2 
+ 2|\nabla |A||^2\VP\, .
\end{align*}

From Kato's inequality, $|\nabla |A||^2 \le |\nabla  A|^2$, so that
\begin{align*}
\text{I} + \text{II} + \text{III} 
\le & \,(\VP(2 + c_1\sigma) -2u^2\VP^\prime  )|A|^4 + \left( 6n+\frac{\sigma}{c_1} + \frac{4}{c_0k}\right)|A|^2\VP + \frac{\sigma}{c_2}\VP
\\
&
 + ( 6 \VP^{-1}(\VP^\prime u)^2 - (6 + 7k\VP)\VP^\prime) |A|^2 |\nabla u|^2
-\VP^{-1}\langle \nabla \VP,\nabla ( |A|^2\VP)\rangle_H \,.
\end{align*}
Note that since $c_1\sigma \le c_0k$ and $\VP - u^2\VP^\prime = -k\VP^2$, we have $\VP(2 + c_1\sigma) -2u^2\VP^\prime  \le -k\VP^2$. Moreover,
$$
6 \VP^{-1}(\VP^\prime u)^2 - (6 + 7k\VP)\VP^\prime \,=\, -k \VP \VP^\prime\,.
$$
Now let $c_1 = \frac{c_0 k }{\sigma}$ and $c_2 = \frac{1}{\sigma}$, then  $6n+\frac{\sigma}{c_1} + \frac{4}{c_0k} \le \frac{c(n,c_0)}{k}$ and on $\{ x\in \Sigma_t | r \le R\} \cap \{ |A|^2 \ge 1\}$, we have
$$\text{I} + \text{II} + \text{III} 
 \le  -k|A|^4\VP^2 
+ \left(\frac{c(n,c_0)}{k} 
- k\VP^\prime|\nabla u|^2\right)|A|^2\VP 
 -\VP^{-1}\langle \nabla \VP,\nabla ( |A|^2\VP)\rangle_H + \sigma^2 \VP\,.$$
This proves the lemma.
\end{proof}

Now we are ready to show the interior estimates on the second fundamental form $|A|$ (i.e., $|\nabla^2 v|$). For simplicity, let $$g = |A|^2\VP\,.$$ 
Then the previous lemma says
\begin{align*}
\left(\frac{\partial}{\partial t} -\Delta \right) g
\le\, -kg^2 + \left(\frac{c(n,c_0)}{k} - k\VP^\prime|\nabla u|^2\right)g -\VP^{-1}\langle \nabla \VP,\nabla g\rangle_H +\sigma^2 \VP\, .
\end{align*}

Now let $$\eta =(\cosh R - \cosh r)^2$$ be the spacial cut-off function, and let $\eta^\prime$ denote the differentiation with respect to $\cosh r$. Then, from Proposition \ref{prop2}, we have
\begin{align*}
\left(\frac{\partial}{\partial t} - \Delta \right)(-\cosh r)= &\,-\left[\frac{1}{\cosh r}(1 - \langle \vnu_E, \mathbf{z}\rangle_E^2) - (n - \sigma\langle \vnu_E, \ve\rangle_E)\cosh r - \sigma\langle \vnu_E, \mathbf{z}\rangle_E\right]\\
\leq &\, (\sigma+n)\cosh r + \sigma\,.
\end{align*}
So that
\begin{align*}
\left(\frac{\partial}{\partial t} -\Delta \right)\eta = &\, 2 (\cosh R - \cosh r)\left( \frac{\partial}{\partial t} - \Delta \right)(-\cosh r) - 2 |\nabla \cosh r|^2\\
\le & \, 2(\sigma + n)\cosh^2 R + 2\sigma \cosh R  - 2 |\nabla \cosh r|^2
\\
\le &\, 2(2\sigma + n )\cosh^2 R- 2 |\nabla \cosh r|^2\,,
\end{align*}
if $\sigma \leq  \cosh R$, namely, $R$ is sufficiently large, e.g., $\cosh R \geq n$.

Therefore, we compute:
\begin{align}
\left(\frac{\partial}{\partial t} -\Delta \right) (g\eta) \leq &\, \left[-kg^2 + \left(\frac{c(n,c_0)}{k} - k\VP^\prime|\nabla u|^2\right)g -\VP^{-1}\langle \nabla \VP,\nabla g\rangle_H +\sigma^2 \VP\right]\eta \notag\\
&+ g \left(\frac{\partial}{\partial t} -\Delta \right)\eta - 2 \langle \nabla g, \nabla \eta\rangle \notag\\
\le &\,-kg^2\eta + \left(\frac{c(n,c_0)}{k} \right)g\eta -\VP^{-1}\langle \nabla \VP,\nabla (g\eta)\rangle_H +\frac{|\eta^\prime|^2 g}{k\eta u^2}|\nabla \cosh r|^2 \notag\\
&+ \sigma^2 \VP \eta+ g \left(\frac{\partial}{\partial t} -\Delta \right)\eta - 2\eta^{-1} \langle \nabla (g\eta), \nabla \eta\rangle + 2\eta^{-1}g|\nabla \eta|^2 \notag\\
\leq &\, -kg^2\eta + \left(\frac{c(n,c_0)}{k} \right)g\eta -\langle \VP^{-1}\nabla \VP + 2\eta^{-1}\nabla \eta,\nabla (g\eta)\rangle_H \notag\\
&+ \sigma^2 \VP \eta+ g \left(2(2\sigma + n )\cosh^2 R- 2 |\nabla \cosh r|^2\right) + g |\nabla \cosh r|^2\left( \frac{4}{ku^2} + 8\right) \notag\\
\leq &\,  -kg^2\eta + \left(\frac{c(n,c_0)}{k} \right)g\eta -\langle \VP^{-1}\nabla \VP + 2\eta^{-1}\nabla \eta,\nabla (g\eta)\rangle_H \\
&+ 30n g \left(1 +\frac{|x|_E^2}{k}\right) \cosh^2R + \sigma^2 \VP \eta \notag\,,
\end{align}
where we used Young's inequality and the facts that $\VP^{-1} \nabla \VP = 2\VP u^{-3} \nabla u$ and $\VP^\prime = \VP^2 u^{-4}$ and $\eta^{-1}|\nabla \eta|^2 = \eta^{-1}|\eta^\prime|^2 |\nabla \cosh r|^2 = 4|\nabla \cosh r|^2 \leq 4 (1+ \cosh r)^2 $\,. Therefore, we have
\begin{align}
\left(\frac{\partial}{\partial t} -\Delta \right) (g\eta t) \leq &\, -kg^2\eta t + \left(\frac{c(n,c_0)}{k} t + 1\right)g\eta -\langle \VP^{-1}\nabla \VP + 2\eta^{-1}\nabla \eta,\nabla (g\eta t)\rangle_H \notag\\
&+ 30n g \left(1 +\frac{1}{c_0k}\right) (\cosh^2R) t + \sigma^2 \VP \eta t\,.
\end{align}

Now at a point $(x_0,t_0)$ where $\sup_{[0,T]}\sup_{\{x\in \Sigma_t | r \le R\}} (g\eta t) \ne 0$ is attained for $t_0>0$, we have
\begin{align*}
kg^2\eta t_0
\le  \left(\frac{c(n,c_0)}{k}t_0 + 1\right)g\eta +30n g \left(1+ \frac{1}{c_0k}\right) (\cosh^2 R)t_0 + \sigma^2 \VP \eta t_0,
\end{align*}
which implies (dividing by $kg = k |A|^2 \VP$ on both sides) at $(x_0, t_0)$ we have
\begin{align*}
&g(x_0,t_0)\eta(x_0,t_0) t_0\\
\leq&\,  \frac{1}{k}\left(\frac{c(n,c_0)}{k}t_0 + 1\right)\cosh^2 R +\frac{30n}{k} \left(1+ \frac{1}{c_0k}\right) (\cosh^2 R)t_0 + \frac{\sigma^2}{k|A|^2} (\cosh^2R) t_0\\
 \leq&\, \frac{c(n,c_0)}{k^2}(1+T)\cosh^2R + \frac{30n}{k}\left(1+ T + \frac{\sigma^2 T}{|A|^2(x_0,t_0)}\right)\cosh^2R\,.
\end{align*}
Note that for any $(x,t) \in \{x\in \Sigma_t | \cosh r\leq \theta \cosh R\} \times [0,T]$ we have
$$
g(x,t)\eta(x,t) t \leq g(x_0,t_0)\eta(x_0,t_0) t_0 \quad\text{and}\quad \eta\geq (1-\theta)^2\cosh^2 R\,.
$$
If $|A|^2(x_0,t_0) \leq 1$, then 
\begin{align*}
c_0 |A|^2(x,T) &\leq \frac{1}{T} \eta^{-1}(x,T) \VP(x_0, t_0) \eta(x_0, t_0) t_0 \\
&\leq 4(1-\theta)^{-2}\sup_{t\in [0,T]}\sup_{\{ x\in \Sigma_t | r \le R\}} u^2\\
& \leq \frac{8}{c_0}(1-\theta)^{-2}\sup_{t\in [0,T]}\sup_{\{ x\in \Sigma_t | r \le R\}} u^4\,,
\end{align*}
where we used $c_0\leq \VP \leq 2 u^2$ and $\eta \leq 2 \cosh^2 R$\,. Otherwise, if $|A|^2(x_0,t_0) > 1$ then we have
\begin{align*}
c_0|A|^2(x,T) \leq g(x,T) &\leq \left[ \frac{c(n,c_0)}{k^2}\left(1+\frac{1}{T}\right)+ \frac{30n}{k}\left(1+ \frac{1}{T} + \sigma^2\right)\right](1-\theta)^{-2}\\
&\leq c(n,c_0)\left(1+\frac{1}{T}\right)(1-\theta)^{-2}\sup_{t\in [0,T]}\sup_{\{ x\in \Sigma_t | r \le R\}} u^4\,.
\end{align*}

Since $T>0$ was arbitrary, we have just proved
\begin{theorem}\label{estonA2}
For all $t \in [0,T]$, any $R\geq \cosh^{-1}(n)$ and any $\theta \in (0,1)$ we have
\begin{align*}
\sup_{\{ x\in \Sigma_t | \cosh r \le \theta \cosh R\}}|A|^2
&
\le   c(n,c_0)\left(1+\frac{1}{t}\right)(1-\theta)^{-2}\sup_{s\in [0,t]}\sup_{\{ x\in \Sigma_s | r \le R\}} u^4\,.\end{align*}
\end{theorem}

\subsection{Estimates on all the higher order derivatives}
The estimates on all the higher order derivatives could be obtained analogously by looking at the evolution equations of the higher derivatives of the second fundamental form, see e.g. \cite{EH91} and \cite{U03}. For this, we have
\begin{lemma} For hypersurfaces $\Sigma_t$ moving by MMCF in $\mathbb{H}^{n+1}$ which can be written locally as radial graphs, we have\\
(i)
\begin{align*}
\left(\frac{\partial}{\partial t} -\Delta\right){\nabla}^m A
& 
=\sum_{i + j + k = m}{\nabla}^i A\ast {\nabla}^j A\ast {\nabla}^k A + \sigma\sum_{i + j = m}{\nabla}^i A\ast {\nabla}^j A 
\\
&
+ \sum_{i +j = m} {\nabla}^i A\ast {\nabla}^j R^H + \sigma\ast{\nabla}^m R^H.
\end{align*}
where $S \ast T$ is a tensor formed  by contraction of tensors $S$ and $T$ by the metric $g$ on $\Sigma_t$ or its inverse\,;\\
(ii)
\begin{align*}
\left(\frac{\partial}{\partial t} -\Delta\right)| {\nabla}^m A|^2
&\le  - 2| {\nabla}^{m+1} A|^2
 + c\,\Bigg(\sum_{i + j + k = m}| {\nabla}^i A|| {\nabla}^j A|| {\nabla}^k A|| {\nabla}^m A| \\
 &+\sigma\sum_{i + j  = m}| {\nabla}^i A|| {\nabla}^j A|| {\nabla}^m A|+| {\nabla}^m A|^2+\sigma | {\nabla}^m A|^2\Bigg)\,.
\end{align*}
\end{lemma}

\begin{theorem}\label{estonA3}
For all $t \in [0,T]$, any $R\geq \cosh^{-1}(n)$ and any $\theta \in (0,1)$ we have
\begin{align*}
\sup_{\{ x\in \Sigma_t | \cosh r \le \theta \cosh R\}} | {\nabla}^m A|^2
&
\le c\left(n,c_0, \sup_{s\in [0,t]}\sup_{\{ x\in \Sigma_s | r \le R\}} u\right)\left(1+\frac{1}{t}\right)(1-\theta)^{-2}\left(1 + \frac{1}{t}\right)^{m+1}\,.\end{align*}
\end{theorem}
\begin{proof}
Similar to the proof of Theorem \ref{estonA2}, c.f. \cite{EH91}\,.
\end{proof}

\section{Proof of Theorem \ref{mainthem}}\label{proofofmain}

Our goal in this section is to prove the main Theorem \ref{mainthem}.

\begin{proof}
We will use the method of continuity. First assume $\Sigma_0$ (or equivalently $v_0$) is smooth. For any $\VE > 0$, define the solid cylinder 
$$\mathbf{C}_{\VE} = \left\{x \in \mathbb{H}^{n+1}:  \frac{|x|_E}{x_{n+1}} \le \frac{1}{\VE}\right\}\,,$$ 
and let $\Sigma_0^\VE = \Sigma_0 \cap \mathbf{C}_{\VE}$ and $\Omega_{\VE} := \mathbf{F}_0^{-1}(\Sigma_0 \cap \mathbf{C}_{\VE})$. Then $\Omega_\VE$ is compact and $\Gamma_\VE := \mathbf{F}_0(\partial\Omega_\VE)$ is a smooth radial graph over $\partial\Omega_\VE$. 

From the existence result in \cite{LX12} for the approximate MMCF we know that the initial-boundary value problem
\begin{equation}\label{MMCF5}
\left\{
\begin{aligned}
\frac{\partial}{\partial t} \mathbf{F}(\mathbf{z},t) &= (H-\sigma)\,\vnu_H\,,\quad (\mathbf{z},t) \in \Omega_{\VE} \times (0,\infty)\,,\\
\mathbf{F}(\mathbf{z},0) &= \mathbf{F}_0(\mathbf{z})\,,\quad \mathbf{z} \in \Omega_{\VE} \,, \\
\mathbf{F}(\mathbf{z},t) & = \Gamma_\VE(\mathbf{z})\, , \quad (\mathbf{z},t)\in \partial \Omega_{\VE} \times [0,\infty)
\end{aligned}
\right.
\end{equation}
has a unique radial graph solution $\mathbf{F}_t^{\VE}(\mathbf{z}) = \mathbf{F}^{\VE}(\mathbf{z},t) \in C^\infty( \Omega_{\VE}\times (0,\infty)) \cap C^{0+1,0+\frac{1}{2}}(\overline{\Omega_{\VE}}\times (0,\infty)) \cap C^0(\overline{\Omega_{\VE}}\times [0,\infty))$, and we denote $\Sigma_t^\VE = \mathbf{F}^\VE(\Omega_\VE, t)$.

Now, for every $\VE \in (0,1)$, let $v^\VE(\vz,t)$ be the solution to \eqref{MMCF5} (c.f. \eqref{MMCF1}), namely,
\begin{equation}\label{MMCF6}
\left\{
\begin{aligned}\frac{\partial v^\VE(\mathbf{z},t)}{\partial t} &=\, y^2\frac{\alpha^{ij}v^\VE_{ij}}{n}-y\mathbf{e}\cdot\nabla
v^\VE-\sigma yw^\VE \,, \quad (\mathbf{z},t)\in \Omega_{\VE} \times (0,\infty)\,,\\
v^\VE(\mathbf{z},0) &= v_0(\mathbf{z})\,,\quad \mathbf{z} \in \Omega_{\VE}\,,\\
v^\VE(\mathbf{z},t) &= \phi^\VE(\mathbf{z})\,,\quad (\mathbf{z},t)\in \partial \Omega_{\VE} \times [0,\infty)\,.
\end{aligned}
\right.
\end{equation}
For a fixed $\delta_0>0$ sufficiently small, let
$$E_{t,\VE,\delta_0} := \Sigma_t^\VE \cap \left\{x \in \mathbb{H}^{n+1} \mid r(x) \le \cosh^{-1}\left(\frac{1}{\delta_0} \right) \right\} = \Sigma_t^\VE \cap \mathbf{C}_{\delta_0}\,,$$
where $r(x)$ is the hyperbolic distance from $x \in \mathbb{H}^{n+1}$ to the $x_{n+1}$-axis and $\cosh r(x) = \frac{|x|_E}{x_{n+1}}$. Then $E_{t,\VE,\delta_0}$ is a compact radial graph and we have $E_{0,\VE,\delta_0} = E_{0,\delta_0,\delta_0}$ for all $\VE\le \delta_0$. By compactness, there exist caps $S_1,S_2$ with constant mean curvature $\sigma$ such that the Euclidean norms satisfy $c^{-1}(\Sigma_0^{\delta_0})\leq |x_1|_E \leq |\mathbf{F}^\VE_0(\vz)| \leq |x_2|_E \leq c(\Sigma_0^{\delta_0})$ for all $x_i \in S_i$, $i = 1,2$, any $\vz \in (\mathbf{F}^\VE_0)^{-1}(E_{0,\VE,\delta_0})$, and any $\VE \le \delta_0$. This implies, by the comparison principle for MMCF, that for all $\VE \le \delta_0$ we have
\begin{equation*}
\sup_{t \in (0,\infty)}\sup_{\vz \in (\mathbf{F}^\VE_t)^{-1}(E_{t,\VE,\delta_0})} |v^\VE(\vz,t)| \le c_0\left(n,\delta_0,\sup_{ \vz \in  \mathbf{F}_0^{-1}(E_{0,\delta_0,\delta_0}) }|v_0(\vz)|\right)\,.
\end{equation*}
For $\theta \in (0,1)$, let 
$$G_{t,\VE,\delta_0,\theta} := \left\{x \in E_{t,\VE,\delta_0} \mid e^{(n+\sigma)t}\left(\cosh r(x) + \frac{\sigma}{n+\sigma}\right)\le  \frac{\theta}{\delta_0}\right\}.$$ 
Note that by Theorem \ref{intgrad}, for all $\VE \le \delta_0$ and any $T_0 > 0$ we have
$$\sup_{t \in [0,T_0]}\sup_{\vz \in (\mathbf{F}^\VE_t)^{-1}( G_{t,\VE,\delta_0,\frac{1}{2}} )}|\nabla v^\VE(\vz,t)| \le e^{(n+2)T_0}c_1\left(n, \delta_0,c_0, \sup_{\vz \in \mathbf{F}_0^{-1}(E_{0,\delta_0,\delta_0})}|\nabla v_0(\vz)| \right)\,.$$

For $\VE_0>0$ and $\theta \in (0,1)$, let 
$$K_{t,\VE,\VE_0,\theta} := \left\{x \in E_{t,\VE,\delta_0} \mid\cosh r(x) \le   \frac{\theta}{\VE_0}\right\}.$$
Choose $\delta_0 > 0$ sufficiently small such that $\frac{1}{\delta_0^{1/2}} - \frac{\sigma}{n+\sigma} \geq 2,$ and let  $T_0 = -\frac{1}{2(n+\sigma)}\log \delta_0$ and $\VE_0= \left(\frac{1}{\delta_0^{1/2}} - \frac{\sigma}{n+\sigma}\right)^{-1}$.
Then, for our choices of $\delta_0, T_0,$  $\VE_0$ we know that for any $\VE \le \delta_0$,
$$G_{T_0, \VE,\delta_0,\frac{1}{2}} = K_{T_0, \VE,\VE_0,\frac{1}{2}}\,.$$

Hence, for all $\VE \le \delta_0$, we have
\begin{equation*}
\sup_{t \in [0,T_0]}\sup_{\vz \in (\mathbf{F}^\VE_t)^{-1}( K_{t,\VE,\VE_0,\frac{1}{2}})}|\nabla v^\VE(\vz,t)| \le e^{(n+2)T_0}c_1\left(n, \delta_0,c_0, \sup_{\vz \in \mathbf{F}_0^{-1}(E_{0,\delta_0,\delta_0})}|\nabla v_0(\vz)| \right)\,.
\end{equation*}
Therefore, by Theorem \ref{estonA3}\,, for any integer $m \ge 2$ and any $\VE \le\delta_0$, we have
\begin{equation*}
\sup_{t \in [0,T_0]}\sup_{\vz \in (\mathbf{F}^\VE_t)^{-1}(K_{t,\VE,\VE_0,\frac{1}{2}})}|\nabla^m v^\VE(\vz,t)| \le c_m(n,\delta_0,c_1)\,.
\end{equation*}

Hence, for such fixed $\delta_0>0$, by the Arzel\`{a}-Ascoli Theorem, 
there exists some sequence
$\{\VE_{i,0}\}^\infty_{i = 1}$ such that $\VE_{i,0} \rightarrow 0$ as $i \rightarrow\infty$ and 
such that $v^{\VE_{i,0}}$
converges uniformly in $C^\infty$ to some $v^{\VE_0,T_0} \in C^\infty(\Omega_{2\VE_0} \times [0,T_0])$ as $i \rightarrow \infty$ which solves \eqref{MMCF6}. Now fix a descending sequence $\{\delta_k\}_{k = 0}^\infty$ such that $\delta_k \rightarrow 0$ as $k \rightarrow \infty$. Then define $T_k =  -\frac{1}{2(n+\sigma)}\log \delta_k,$ and $\frac{1}{\VE_k} = \frac{1}{\delta_k^{1/2}} - \frac{\sigma}{n+\sigma}.$ Then $T_k \rightarrow \infty$ and $\VE_k \rightarrow 0$ as $k \rightarrow \infty$.

For non-negative integers $k$, suppose we have a function $v^{\VE_k,T_k} \in C^\infty(\Omega_{2\VE_k} \times [0,T_k])$ solving \eqref{MMCF6} such that $v^{\VE_k,T_k}$ is the uniform limit of some sequence $\{v^{\VE_{i,k}}\}_{i = 1}^\infty$ and  $v^{\VE_k,T_k}\vert_{\Omega_{2\VE_l} \times [0,T_l]} = v^{\VE_l,T_l}$ for all non-negative integers $l \le k$. We can see this by induction. The case of $k = 0$ was done above. Our interior estimates imply we have uniform bounds of $v^{\VE}$ and its derivatives on $\Omega_{2\VE_{k+1}}\times [0,T_{k+1}]$ for $\VE \le \delta_{k+1}$. So, again by the Arzel\`{a}-Ascoli Theorem, there exists a subsequence $\{v^{\VE_{i,k+1}}\}_{i = 1}^\infty$ of $\{v^{\VE_{i,k}}\}_{i = 1}^\infty$ such that $v^{\VE_{i,k+1}}$ converges uniformly to some $v^{\VE_{k+1},T_{k+1}} \in C^\infty(\Omega_{2\VE_{k+1}} \times [0,T_{k+1}])$ as $i \rightarrow \infty$. Since $\Omega_{2\VE_{k}} \times [0,T_{k}] \subset \Omega_{2\VE_{k+1}} \times [0,T_{k+1}]$ and $\{v^{\VE_{i,k+1}}\}_{i = 1}^\infty$ is a subsequence of $\{v^{\VE_{i,k}}\}_{i = 1}^\infty$, we must have $v^{\VE_{k+1},T_{k+1}}\vert_{\Omega_{2\VE_k} \times [0,T_k]} = v^{\VE_k,T_k}$.

If $(\vz,t) \in \mathbb{S}^n_+ \times [0,\infty)$, then there exists some non-negative integer $k$ such that $(\vz, t) \in \Omega_{2\VE_k}\times [0,T_k]$. Define $v(\vz,t) = v^{\VE_k,T_k}(\vz,t)$. Then our construction of the sequence $v^{\VE_k,T_k}$ shows this definition is well-defined. Moreover,  if we define $\mathbf{F}(\vz,t) = e^{v(\vz,t)}\vz$ on $\mathbb{S}^n_+ \times [0,\infty)$,  then $\mathbf{F} \in C^\infty( \mathbb{S}^n_+ \times [0,\infty))$ solves \eqref{MMCF1}.

Now if $\Sigma_0$ is merely locally Lipschitz continuous, then for any fixed compact subset $\Omega \subset \mathbb{S}^n_+$, we can approximate $v_0$ by smooth functions $v^j_0$ with the same Lipschitz bound as the Lipschitz bound of $v_0$ on $\Omega$. By the above arguments, for every $s$, there is a smooth one parameter family of functions $v_t^j$ solving \eqref{MMCF6} with initial data $v_0^s$. Now our interior estimates imply $v_t^j$ and all its derivatives are uniformly bounded in any compact set $K \subset \Omega$, which again implies the existence of a uniform limit $v \in C^\infty(K \times (0,T]) \cap C^{0+1,0+1/2}(K \times [0,T])$. Since $\Omega$ and $T$ were arbitrary, this establishes the existence of a function  $v \in C^\infty(\mathbb{S}^n_+ \times (0,\infty)) \cap C^{0+1,0+1/2}(\mathbb{S}^n_+ \times [0,\infty))$ which solves \eqref{MMCF1}.
\end{proof}

\bibliographystyle{alpha}
\bibliography{MMCF}

\newcommand{\noopsort}[1]{} \newcommand{\singleletter}[1]{#1}
\begin{thebibliography}{CMP15}

\bibitem[And82]{A82}
Michael~T. Anderson.
\newblock Complete minimal varieties in hyperbolic space.
\newblock {\em Invent. Math.}, 69(3):477--494, 1982.

\bibitem[Bar84]{Bar84}
Robert Bartnik.
\newblock Existence of maximal surfaces in asymptotically flat spacetimes.
\newblock {\em Comm. Math. Phys.}, 94(2):155--175, 1984.

\bibitem[BP92]{BP}
Riccardo Benedetti and Carlo Petronio.
\newblock {\em Lectures on hyperbolic geometry}.
\newblock Universitext. Springer-Verlag, Berlin, 1992.

\bibitem[Bra78]{B78}
K.~Brakke.
\newblock {\em The motion of a surface by its mean curvature}, volume~20 of
  {\em Mathematical Notes}.
\newblock Princeton University Press, Princeton, N.J., 1978.

\bibitem[CMP15]{CMP}
Tobias~Holck Colding, William~P. Minicozzi, II, and Erik Kj\ae~r Pedersen.
\newblock Mean curvature flow.
\newblock {\em Bull. Amer. Math. Soc. (N.S.)}, 52(2):297--333, 2015.

\bibitem[DSS09]{DS09}
D.~De~Silva and J.~Spruck.
\newblock Rearrangements and radial graphs of constant mean curvature in
  hyperbolic space.
\newblock {\em Calc. Var. Partial Differential Equations}, 34(1):73--95, 2009.

\bibitem[EH89]{EH89}
Klaus Ecker and Gerhard Huisken.
\newblock Mean curvature evolution of entire graphs.
\newblock {\em Ann. of Math. (2)}, 130(3):453--471, 1989.

\bibitem[EH91]{EH91}
Klaus Ecker and Gerhard Huisken.
\newblock Interior estimates for hypersurfaces moving by mean curvature.
\newblock {\em Invent. Math.}, 105(3):547--569, 1991.

\bibitem[GS00]{GS00}
Bo~Guan and Joel Spruck.
\newblock Hypersurfaces of constant mean curvature in hyperbolic space with
  prescribed asymptotic boundary at infinity.
\newblock {\em Amer. J. Math.}, 122(5):1039--1060, 2000.

\bibitem[HL87]{HL87}
Robert Hardt and Fang-Hua Lin.
\newblock Regularity at infinity for area-minimizing hypersurfaces in
  hyperbolic space.
\newblock {\em Invent. Math.}, 88(1):217--224, 1987.

\bibitem[HLZ16]{HLZ}
Zheng Huang, Longzhi Lin, and Zhou Zhang.
\newblock Mean curvature flow in fuchsian manifolds.
\newblock {\em arXiv:1605.06565, preprint}, 2016.

\bibitem[Hui84]{H84}
G.~Huisken.
\newblock Flow by mean curvature of convex surfaces into spheres.
\newblock {\em J. Differential Geom.}, 20(1):237--266, 1984.

\bibitem[Hui86]{H86}
Gerhard Huisken.
\newblock Contracting convex hypersurfaces in {R}iemannian manifolds by their
  mean curvature.
\newblock {\em Invent. Math.}, 84(3):463--480, 1986.

\bibitem[Hui90]{H90}
G.~Huisken.
\newblock Asymptotic behavior for singularities of the mean curvature flow.
\newblock {\em J. Differential Geom.}, 31(1):285--299, 1990.

\bibitem[Lin89]{Lin89}
Fang-Hua Lin.
\newblock On the {D}irichlet problem for minimal graphs in hyperbolic space.
\newblock {\em Invent. Math.}, 96(3):593--612, 1989.

\bibitem[LX12]{LX12}
Longzhi Lin and Ling Xiao.
\newblock Modified mean curvature flow of star-shaped hypersurfaces in
  hyperbolic space.
\newblock {\em Comm. Anal. Geom.}, 20(5):1061--1096, 2012.

\bibitem[NS96]{NS96}
Barbara Nelli and Joel Spruck.
\newblock On the existence and uniqueness of constant mean curvature
  hypersurfaces in hyperbolic space.
\newblock In {\em Geometric analysis and the calculus of variations}, pages
  253--266. Int. Press, Cambridge, MA, 1996.

\bibitem[Ton96]{T96}
Yoshihiro Tonegawa.
\newblock Existence and regularity of constant mean curvature hypersurfaces in
  hyperbolic space.
\newblock {\em Math. Z.}, 221(4):591--615, 1996.

\bibitem[Unt03]{U03}
Philip Unterberger.
\newblock Evolution of radial graphs in hyperbolic space by their mean
  curvature.
\newblock {\em Comm. Anal. Geom.}, 11(4):675--695, 2003.

\end{thebibliography}

\end{document}